\documentclass[final,3p,times]{elsarticle}

\usepackage{amssymb}
\usepackage{amsthm}

\usepackage{amsthm,amsfonts,amsmath,verbatim,amssymb,verbatim}

\usepackage[usenames]{color}
\usepackage{tikz,graphicx}

\numberwithin{equation}{section}
\newtheorem{theorem}{Theorem}[section]

\newtheorem{lemma}[theorem]{Lemma}
\newtheorem{claim}[theorem]{Claim}
\newtheorem{corollary}[theorem]{Corollary}
\newtheorem{conjecture}[theorem]{Conjecture}
\newtheorem{example}[theorem]{Example}

\def\Z{\mathbb Z}
\def\S{\mathcal S}
\def\D{\mathcal D}
\def\Ce{\mathcal C}
\def\For{\mathcal F}
\def\L{\mathcal L}
\def\K{\mathcal K}
\def\M{\mathcal M}
\def\A{\mathcal A}
\def\Fo{\mathcal F}

\def\Q{\mathbb Q}
\def\F{\mathbb F}

\def\De{\Delta}
\def\V{\mathfrak V}
\def\B{\mathfrak B}

\def\supp{\mbox{\rm supp}}
\def\sign{\mbox{\rm sign}}
\def\char{\mbox{\rm char}}

\newcommand{\abs}[1]{\lvert#1\rvert}
\newcommand{\Symm}{\mathfrak{S}}
\newcommand{\qbin}[2]{\genfrac{[}{]}{0pt}{}{#1}{#2}}
\newcommand{\An}{\abs{\ba}}
\newcommand{\ba}{\boldsymbol{a}}
\newcommand{\bx}{\boldsymbol{x}}
\newcommand{\bc}{\boldsymbol{c}}
\newcommand{\bB}{\boldsymbol{B}}

\DeclareMathOperator{\id}{id}
\DeclareMathOperator{\CT}{CT}
\DeclareMathOperator{\Hom}{Hom}

\begin{document}

\begin{frontmatter}



\title{A new approach to constant term identities and Selberg-type integrals}


\author{Gyula K\'arolyi\corref{cor1}}
\cortext[cor1]{Corresponding author. 
On leave from E\"otv\"os University, Budapest.}
\address{School of Mathematics and Physics,
The University of Queensland, Brisbane, QLD 4072, Australia}
\ead{karolyi@cs.elte.hu}

\author{Zolt\'an L\'or\'ant Nagy}
\address{Alfr\'ed R\'enyi Institute of Mathematics,
Re\'altanoda utca 13--15, Budapest, 1053 Hungary}
\ead{nagyzoltanlorant@gmail.com}

\author{Fedor V. Petrov\fnref{label2,label3}}
\fntext[label2]{Also at St. Petersburg State University.}
\fntext[label3]{Also at B.N. Delone International Laboratory 
``Discrete and Computational Geometry".}
\address{Steklov Institute of Mathematics,
Fontanka 27, 191023 St. Petersburg, Russia}
\ead{fedyapetrov@gmail.com}

\author{Vladislav Volkov\fnref{label3}}
\address{Saint-Petersburg State University, 
Universitetsky prospekt 28, 
198504 St. Petersburg, Russia}
\ead{vladvolkov239@gmail.com}

\begin{abstract}
Selberg-type integrals that can be turned into constant term identities
for Laurent polynomials arise naturally in conjunction with random matrix 
models in statistical mechanics. 
Built on a recent idea of Karasev and Petrov we develop
a general interpolation based method
that is powerful enough to establish many such identities in a simple manner.
The main consequence is the proof of a conjecture of Forrester related to the
Calogero--Sutherland model. In fact we prove a more general theorem,
which includes Aomoto's constant term identity at the same time.
We also demonstrate the relevance of the method in additive combinatorics.
\end{abstract}


\end{frontmatter}



\bigskip
\section{Introduction}
\label{1}

Perhaps the most famous constant term identity is the one associated with
the name of Freeman Dyson. In his seminal paper \cite{Dyson} dated back to
1962, Dyson proposed to replace Wigner's classical Gaussian-based random
matrix models by what now is known as the circular ensembles. The study of
their joint eigenvalue probability density functions led Dyson to the
following conjecture. Consider the family of Laurent polynomials
\[
\D(\bx;\ba):=\prod_{1\leq i \neq j \leq n}\left(1-\frac{x_i}{x_j}\right)^{a_i}
\]
parametrized by a sequence $\ba=(a_1,\dots,a_n)$ of nonnegative integers,
where $\bx=(x_1,\dots,x_n)$ is a sequence of indeterminates.
Denoting by $\CT[\L(\bx)]$ the constant term of the Laurent polynomial 
$\L=\L(\bx)$, Dyson's hypothesis can be formulated as the identity
\[
\CT[\D(\bx;\ba)]=\frac{(a_1+a_2+ \dots +a_n)!}{a_1!a_2! \dots a_n!}=:
\binom{\abs{\ba}}{\ba},
\]
where $\abs{\ba}=a_1+a_2+\dots+a_n$.
Using the shorthand notation $\D(\bx;k)$ for the equal parameter case
$\ba=(k,\dots,k)$, the constant term of $\D(\bx;k)$ for $k=1,2,4$ 
corresponds to the normalization factor of the partition function for the
circular orthogonal, unitary and symplectic ensemble, respectively.

Dyson's conjecture was confirmed by Gunson
[unpublished]\footnote{Gunson's proof is similar
to that of Wilson, cf. \cite{Dyson}.
A related conjecture of Dyson is proved by Gunson in \cite{Gunson}.} and Wilson 
\cite{Wilson} in the same year. The most 
elegant proof, based on Lagrange interpolation, is due to Good \cite{Good}.

Let $q$ denote yet another independent variable. In 1975
Andrews \cite{Andrews} suggested the following $q$-analogue 
of Dyson's conjecture: The constant term 
of the Laurent polynomial
\[
\D_q(\bx;\ba):=\prod_{1\leq i < j \leq n} 
\left(\frac{x_i}{x_j}\right)_{a_i}\left(\frac{qx_j}{x_i}\right)_{a_j}
\]
must be the $q$-multinomial coefficient
$$\qbin{\An}{\ba}:=\frac{\left(q\right)_{\abs{\ba}}}
{\left(q\right)_{a_1}\left(q\right)_{a_2}\dots\left(q\right)_{a_n}},$$
where $\big(t\big)_{k}= (1-t)(1-tq)\dots(1-tq^{k-1})$.  
Note that the slight asymmetry of the function $\D_q$ disappears when
one considers $\D=\D_1$; 
specializing at $q=1$, Andrews' conjecture gives back that of Dyson.

Despite several attempts \cite{Kadell1,Stanley1,Stanley2} the problem 
remained unsolved until 1985, when Zeilberger and Bressoud \cite{ZB} 
found a tour de force combinatorial proof; see also \cite{BG}.
Shorter proofs are due to
Gessel and Xin \cite{GX} and Cai \cite{Cai}. Recently
an idea of Karasev and Petrov \cite{KP} led to a very short proof 
by K\'arolyi and Nagy \cite{KN}, which we consider as a precursor to
the present paper. 
\medskip

Constant term identities like these and their generalizations are
intimately related to Selberg's integral formula \cite{Selberg}.
Colloquially referred to as the Selberg integral, it asserts
\begin{align*}
S_n(\alpha,\beta,\gamma)&:=
{\int}_0^1\dots{\int}_0^1 \prod_{i=1}^nt_i^{\alpha-1}(1-t_i)^{\beta-1}
\prod_{1\le i<j\le n}|t_i-t_j|^{2\gamma} dt_1\dots dt_n\\
&=\prod_{j=0}^{n-1}\frac{\Gamma(\alpha+j\gamma)\Gamma(\beta+j\gamma)
\Gamma(1+(j+1)\gamma)}{\Gamma(\alpha+\beta+(n+j-1)\gamma)\Gamma(1+\gamma)},
\end{align*}
where the complex parameters $\alpha,\beta,\gamma$ satisfy
\[
\Re(\alpha)>0,\ \Re(\beta)>0,\ \Re(\gamma)>-\min\{1/n, \Re(\alpha)/(n-1),
\Re(\beta)/(n-1)\}.
\]
The continued interest in the Selberg integral, demonstrated 
for example by the most
recent article \cite{Ostrovsky}, is due to its role in random matrix
theory, statistical mechanics, special function theory among other fields;
see the comprehensive exposition \cite{FW}.

The Selberg integral is well-known to be equivalent to Morris's constant
term identity \cite{Morris}
\begin{equation}
\CT\Bigg[ \prod_{j=1}^n(1-x_j)^a(1-1/x_j)^b \D(\bx;k) \Bigg]=
\prod_{j=0}^{n-1}\frac{(a+b+kj)!(kj+k)!}{(a+kj)!(b+kj)!k!},
\label{morris}
\end{equation}
or in a more compact form,
\begin{equation*}
\CT\left[ \M(\bx;a,b,k) \right] = M(n;a,b,k),
\end{equation*}
where the parameters $a,b,k$ are nonnegative integers. The equivalence
is established, via a suitable change of variables, by an application
of a theorem of Carlson \cite{Carlson} and the residue theorem.
This method can be employed to reduce Selberg-type integrals
to constant term identities.

Introducing an extra $t_1\cdots t_m$ factor into the integrand,
Aomoto \cite{Aomoto} in 1987 proved an extension of the Selberg integral.
Based on the fundamental theorem of calculus, it yields
besides Anderson's \cite{Anderson} one of the
simplest known proofs of the Selberg integral itself. Turned into a 
constant term identity, Aomoto's integral reads as
\begin{equation}
\CT\Bigg[ \prod_{j=1}^n(1-x_j)^{a+\chi(j\le m)}(1-1/x_j)^b \D(\bx;k) \Bigg]=
\prod_{j=0}^{n-1}\frac{(a+b+kj+\chi(j\ge n-m))!
(kj+k)!}{(a+kj+\chi(j\ge n-m))!(b+kj)!k!},
\label{aomoto}
\end{equation}
where $\chi(S)$ is equal to 1 if the statement $S$ is true and 0 otherwise.

\medskip
Intimately related to the theory of random matrices, in particular
the Dyson Brownian motion model \cite{Dyson2},
is the Calogero-Sutherland quantum many body system for
spinless quantum particles on the unit circle interacting via the
$1/r^2$ two-body potential, see \cite[Chapter 11]{Forrester2}.
Generalizations to include internal degrees of freedom of the particles
were formulated in the early 1990's. In his 1995 paper \cite{Forrester}
Forrester initiated the study of the analogue of the Selberg integral for
the corresponding exact multicomponent ground-state wavefunction.
Presented in the form of the constant term for the Laurent polynomial
\[
\For(\bx;n_0;a,b,k)=
\M(\bx;a,b,k) \prod_{n_0<i\ne j\le n}\left(1-\frac{x_i}{x_j}\right),
\]
the normalization factor for the most interesting two-component case 
can be determined by the conjectured identity
\begin{equation*}
\CT\left[ \For(\bx;n_0;a,b,k) \right] = M(n_0;a,b,k)\times
\prod_{j=0}^{n-n_0-1}\frac{(j+1)(a+b+kn_0+(k+1)j)!(kn_0+(k+1)j+k)!}{(a+kn_0+
(k+1)j)!(b+kn_0+(k+1)j)!k!}.
\end{equation*}
A $q$-analogue of this hypothesis which extends the $q$-Morris ex-conjecture 
\cite{Morris} was formulated and studied in \cite{BF}. Despite several further 
attempts \cite{Baratta,GLXZ,Hamada,Kaneko1,Kaneko2,Kaneko3,Kaneko4}, 
these conjectures have been resolved only in some particular cases.
The main achievement in the present paper is the proof of these identities,
and in a form that also includes Aomoto's formula (\ref{aomoto}); see Theorem
\ref{a-f} for the precise formulation. 
Along the way we develop a method 
with a wide range of possible applications, some of which are given as
instructive examples.

\medskip
A new proof of the Dyson conjecture given in \cite{KP} and the 
subsequent proof of the Zeilberger--Bressoud identity
presented in \cite{KN} are based on a quick application of 
the following explicit version of the
Combinatorial Nullstellensatz \cite{Alon} found independently by
Laso\'n \cite{Lason} and by Karasev and Petrov \cite{KP}.

\begin{lemma}
Let $\mathbb{F}$ be an arbitrary field and 
$F\in \mathbb{F}[x_1, x_2, \dots, x_n]$ a polynomial of degree
$\deg(F)\leq d_1+d_2+\dots+d_n$. 
For arbitrary subsets $C_1, C_2, \dots, C_n$ of $\mathbb{F}$ 
with $|C_i|=d_i+1$, the coefficient of $\prod x_i^{d_i}$ in $F$ is
$$ \sum_{c_1\in C_1} \sum_{c_2\in C_2} \dots \sum_{c_n\in C_n} 
\frac{F(c_1, c_2, \dots, c_n)}{\phi_1'(c_1)\phi_2'(c_2)\dots \phi_n'(c_n)},$$
where $\phi_i(z)= \prod_{c\in C_i}(z-c)$.
\label{interpol}
\end{lemma}

\noindent
One principal aim of the present paper is to turn this idea into a method,
which has the power to reduce 
seemingly difficult evaluations to simple combinatorial problems.
To this end, in the next section we present a somewhat abstract framework,
which allows us to extend the previous lemma to multisets via Hermite 
interpolation. In Section 3 we demonstrate the strength of the method
in additive combinatorics by providing a new proof of an extension of 
the Erd\H os--Heilbronn conjecture, which is devoid of the heavy technical
details that were needed previously. This is followed in Section 4 by
an application to a problem of Kadell \cite{Kadell2} in algebraic
combinatorics, where the amount of reduction of former complexities is
even more voluminous. In Section 5, which can be viewed as a prelude
to the main result, we reestablish (\ref{morris}) using
our method, thereby giving a short proof of the Selberg integral itself.
Besides formulating our main result, in Section 6 we point out how a 
slight modification yields, modulo some routine computation, a one-page
derivation of the $q$-Morris identity. The same
idea with more delicate combinatorics leads to the solution of the problem
of Forrester in the concluding section. Finally we mention that the method
developed here can be successfully applied to prove Kadell's orthogonality
conjectures \cite{Kadell3}, see \cite{KLW}.

\bigskip
\section{On the Combinatorial Nullstellensatz}
\label{2}

Alon's Nullstellensatz \cite{Alon} describes effectively the structure 
of polynomials which vanish on a finite Cartesian product over an arbitrary
field. It implies the following non-vanishing criterion. Let $F$ be a 
polynomial as in Lemma \ref{interpol}. If the coefficient of $\prod x_i^{d_i}$ 
in $F$ is non-zero, then $F$ cannot vanish on a set $C_1\times C_2\times\dots
\times C_n$, where $|C_i|>d_i$ for every $i$. Note that this is also an
immediate consequence of  Lemma \ref{interpol}. A standard application
of the polynomial method to prove a combinatorial hypothesis works as follows.
Assuming the falsity of the hypothesis, build a polynomial whose values are
all zero over a large Cartesian product, then compute the coefficient of 
the appropriate leading term. If that coefficient is not zero, the criterion
leads to the desired contradiction. The difficulty often lies in the
computation of that coefficient. This is where the power of Lemma 
\ref{interpol} comes into the picture, which is clearly demonstrated in the
next section. An extension of the non-vanishing criterion for the case
when $C_i$ are multisets, along with some applications, was obtained
recently by K\'os and R\'onyai \cite{KR}; see also \cite{KMR}.
Here we generalize Lemma \ref{interpol} in a similar spirit.
\medskip

Let $\V_1,\dots,\V_n$ be vector spaces over the same field $\F$.
For each $i$, fix a basis $\B_i$ in $\V_i$ and fix the corresponding basis
$\otimes \B_i$ in the tensor product space $\otimes \V_i$.
Consider arbitrary non-empty subsets $A_i\subseteq \B_i$, labelled vectors
$a_i\in A_i$, and linear functionals $\eta_i\in \Hom(\V_i,\F)$ that
satisfy the conditions $\eta_i(a_i)=1$ and
$\eta_i(b)=0$ for every $b\in A_i\setminus\{a_i\}$.
Our tool will be the following straightforward observation.

\begin{lemma}
Assume that the tensor $F\in \otimes \V_i$ 
satisfies the following condition: if $b_i\in {\B}_i$
and the coordinate of $F$ at $\otimes b_i$ does not vanish,
then either $b_i=a_i$ for every $i$ or $b_i\in A_i\setminus \{a_i\}$
for at least one index $i$.
Then the coordinate of $F$ at $\otimes a_i$ equals $(\otimes \eta_i)(F)$.
\hfill \qed
\label{nonsense}
\end{lemma}

\noindent
We will apply this lemma in the following situation: 
\[
\V_i=\F[x_i],\quad \B_i=\{1,x_i,x_i^2,\dots\},\quad 
A_i=\{1,x_i,\dots,x_i^{d_i}\},\quad a_i=x_i^{d_i}.
\]
Moreover we will assume that the value of $\eta_i\in \Hom(\V_i,\F)$
at $f\in \F[x_i]$ is the same as the coefficient of $x_i^{d_i}$ in $f$ if
$\deg(f)\le d_i$.
Now $\F[x_1,\dots,x_n]$, as a vector space over $\F$, can be identified
with $\otimes \V_i$ via the unique isomorphism, which extends the
correspondence
\[
x_1^{k_1}\dots x_n^{k_n} \longleftrightarrow x_1^{k_1}\otimes\dots\otimes x_n^{k_n},
\qquad k_i\in\{0,1,2\dots\}.
\]
An important feature of this identification is the following. 

\begin{lemma}
Assume that linear functionals $\vartheta_i\in \Hom(\V_i,\F)$
are given in the form $\vartheta_i(f)=f^{(m_i)}(c_i)$ for some elements
$c_i\in \F$ and nonnegative integers $m_i$. Then for any polynomial
$G\in \F[x_1,\dots,x_n]$,
\[
(\otimes\vartheta_i)(G)=\frac{\partial^{m_1+\dots+m_n}G}{\partial x_1^{m_1}\dots
\partial x_n^{m_n}}(c_1,\ldots,c_n).
\]
\label{nonsense2}
\end{lemma}

\begin{proof}
Indeed, identifying $\F\otimes\dots\otimes\F$ with $\F$ via
the correspondence $\alpha_1\otimes\dots\otimes\alpha_n \longleftrightarrow
\alpha_1\dots\alpha_n$,
for any monomial $G=x_1^{k_1}\dots x_n^{k_n}\in \otimes\B_i$ we obtain 
\[
(\otimes\vartheta_i)(G)=
\otimes \left(\vartheta_i(x_i^{k_i})\right)=
{\prod}_{i=1}^n k_i(k_i-1)\dots(k_i-m_i+1)\cdot \left( \otimes c_i^{k_i-m_i}\right)
=\frac{\partial^{m_1+\dots+m_n}G}{\partial x_1^{m_1}\dots\partial x_n^{m_n}}
(c_1,\dots,c_n).
\]
The general statement follows by linearity.
\end{proof}

Let $F\in  \F[x_1,\dots,x_n]$. We say that no monomial majorizes
$\prod x_i^{d_i}$ in $F$ if every monomial $\prod x_i^{k_i}$ with a non-zero
coefficient in $F$ satisfies either $k_i=d_i$ for every $i$ or $k_i<d_i$
for some $i$. This is certainly the case if
$\deg(F)\le d_1+\dots+d_n$.
Such a polynomial $F$ obviously satisfies the condition
in Lemma \ref{nonsense}. As a warm-up exercise we reestablish Lemma
\ref{interpol} in a slightly stronger form using this language.

\begin{theorem}
Let $F\in  \F[x_1,\dots,x_n]$ be a polynomial such that
no monomial majorizes $M=\prod x_i^{d_i}$ in $F$.
Let $C_1,\dots,C_n$ be arbitrary subsets of $\F$ such that $|C_i|=d_i+1$
for every $i$. Then the coefficient of $M$ in $F$ can be evaluated as
\[
\sum_{c_1\in C_1} \sum_{c_2\in C_2} \dots \sum_{c_n\in C_n} 
\prod_{i=1}^n \kappa(C_i,c_i) {F(c_1, c_2, \dots, c_n)},
\]
where $\kappa(C_i,c_i)=\left(\prod_{c\in C_i\setminus\{ c_i\}}
(c_i-c)\right)^{-1}$. 
Consequently, if the above coefficient is not zero, then there exists a system
of representatives $c_i\in C_i$ such that $F(c_1, c_2, \dots, c_n)\ne 0$.
\label{L-interpol}
\end{theorem}

\begin{proof}
Define the linear functionals $\eta_i\in \Hom(\V_i,\F)$ by
$\eta_i(f)={\sum}_{c_i\in C_i}\kappa(C_i,c_i)f(c_i)$. According to the Lagrange
interpolation formula, $\eta_i(f)$ is equal to the coefficient of 
$x_i^{d_i}$ in $f$ for any $f\in \F[x_i]$ with $\deg(f)\le d_i$.
Since each $\eta_i$ is a linear combination of linear functionals of the
form $\vartheta_i(f)=f^{(0)}(c_i)$, the claim follows easily from
Lemmas \ref{nonsense} and \ref{nonsense2}.
\end{proof}

Extending the notion of the 0/1-valued characteristic function of a set, 
a finite multiset $C$ in $\F$ can be represented by a multiplicity 
function $\omega:\F\to \{0,1,2,\dots\}$ with finite sum 
$|C|:=\sum_{x\in \F} \omega(x)$. We denote by 
$\supp(C):=\{c\in \F\ |\ \omega(c)\ne 0\}$ the supporting set of $C$
and, with a slight abuse of notation, write $c\in C$ if $c\in \supp(C)$. 
A finite union of multisets
is understood as the sum of the corresponding multiplicity functions.
An appropriate generalization of Theorem \ref{L-interpol} for multisets
can be formulated as follows.

\begin{theorem}
Let $F\in  \F[x_1,\dots,x_n]$ be a polynomial such that
no monomial majorizes $M=\prod x_i^{d_i}$ in $F$.
Let $C_1,\dots,C_n$ be arbitrary multisets in $\F$ with corresponding
multiplicity functions $\omega_1,\dots,\omega_n$ such that $|C_i|=d_i+1$
for every $i$. Assume that either $\char(\F)=0$ or $\char(\F)\ge\omega_i(c)$
for every index $i$ and $c\in \F$.
Then the coefficient of $M$ in $F$ can be evaluated as
\[
[M]F=\sum_{c_1\in{C_1}}\sum_{m_1<\omega_1(c_1)}\dots
\sum_{c_n\in{C_n}}
\sum_{m_n<\omega_n(c_n)}
\prod_{i=1}^n \kappa(C_i,c_i,m_i) 
\frac{\partial^{m_1+\dots+m_n}F}{\partial x_1^{m_1}\dots
\partial x_n^{m_n}}(c_1,\ldots,c_n),
\]
where 
\[
\kappa(C_i,c_i,m_i)=\frac{1}{m_i! \cdot (\omega_i(c_i)-1-m_i)!} \cdot 
\left.
\left(\frac{1}{\prod_{c \in {C_i} \setminus\{c_i\}}(x-c)^{\omega_i(c)}}
\right)^{(\omega_i(c_i)-1-m_i)}  
\right|_{x=c_i}\ .
\]
Consequently, if $[M]F\ne 0$,
then there exists a system
of representatives $c_i\in C_i$ and multiplicities $m_i<\omega_i(c_i)$
such that 
\[
\frac{\partial^{m_1+\dots+m_n}F}{\partial x_1^{m_1}\dots
\partial x_n^{m_n}}(c_1,\ldots,c_n)\ne 0.
\]
\label{H-interpol}
\end{theorem}

\noindent
{\em Remarks.} {\em 1.} We tacitly assume that the $m_i$'s are nonnegative 
integers. 
{\em 2.} When
each $\omega_i$ is a 0/1-valued function, the statement reduces to Theorem
\ref{L-interpol}.
{\em 3.} It is possible to derive this result, in a slightly weaker 
form, from the earlier works of K\'os et al. \cite{KMR,KR}. We preferred
this more direct approach.

\begin{proof}
To construct the linear functionals $\eta_i$ we replace Lagrange 
interpolation by Hermite interpolation. For $c_i\in {C_i}$,
$0\le m_i<\omega_i(c_i)$, let 
$g(C_i,c_i,m_i)$ denote the unique polynomial of degree less than $|C_i|$,
provided by the Chinese Remainder Theorem, to the system of simultaneous
congruences
\[
g(C_i,c_i,m_i)(x_i)\equiv (x_i-c_i)^{m_i}/m_i! \pmod{(x_i-c_i)^{\omega_i(c_i)}},
\]
\[
g(C_i,c_i,m_i)(x_i)\equiv 0 \pmod{(x_i-c)^{\omega_i(c)}}\quad 
(c\in {C_i}\setminus\{c_i\})
\]
in $\V_i$. That is, $g(C_i,c_i,m_i)$ is the unique polynomial 
$g\in \F[x_i]$ of 
degree less than or equal to $d_i$, which satisfies $g^{(m_i)}(c_i)=1$ and 
$g^{(m')}(u)=0$ otherwise if $m'<\omega_i(u)$, $u\in \F$. Denote by 
$\kappa(C_i,c_i,m_i)$ the coefficient of $x_i^{d_i}$ in $g(C_i,c_i,m_i)$. Then
Lemmas  \ref{nonsense} and \ref{nonsense2} can be applied as before
for the linear functionals $\eta_i\in \Hom(\V_i,\F)$ given by
\[
\eta_i(f)=\sum_{c_i\in {C_i}}\sum_{m_i<\omega_i(c_i)} \kappa(C_i,c_i,m_i) 
f^{(m_i)}(c_i).
\]
To compute the coefficients $\kappa(C_i,c_i,m_i)$, write 
$p_i(x_i)=\prod_{c\in {C_i}\setminus\{c_i\}}(x_i-c)^{\omega_i(c)}$. That is,
there exist polynomials $h_i,r_i\in\F[x_i]$ with $\deg(h_i)<\omega_i(c_i)$ and
$\deg(r_i) \le d_i-\omega_i(c_i)$ such that
\[
h_i(x_i)=\frac{g(C_i,c_i,m_i)(x_i)}{p_i(x_i)}=\frac{(x_i-c_i)^{m_i}}{m_i!p_i(x_i)}+
(x_i-c_i)^{\omega_i(c_i)}\frac{r_i(x_i)}{p_i(x_i)}
\]
with $\kappa(C_i,c_i,m_i)$ being the coefficient of $x_i^{\omega_i(c_i)-1}$ in
$h_i(x_i)$. Expanding both the left- and the right-hand side as a 
formal power series in the variable $x_i-c_i$ one finds that 
$\kappa(C_i,c_i,m_i)$
is the coefficient of $(x_i-c_i)^{\omega_i(c_i)-m_i-1}$ in $1/(m_i!p_i(x_i))$. 
The result follows by an application of Taylor's formula. 
\end{proof}

\bigskip
\section{An application to additive theory}
\label{3}

Let $\S=\{S_{ij}\mid 1\le i<j\le n\}$ be a family of subsets 
of the cyclic group $\Z_p:=\Z/p\Z$ of prime order $p$.
For a collection of sets $A_1,\dots,A_n\subseteq \Z_p$, consider the
following restricted sumset:
\[
{\bigwedge}_{\S}A_i=\left\{a_1+\dots+a_n\mid a_i\in A_i,\ a_j-a_i\not\in S_{ij}\ 
\mathrm{for}\ i<j\right\}.
\]
For the special case when $A_i\equiv A$ and $S_{ij}\equiv\{0\}$, Dias da
Silva and Hamidoune \cite{DH} proved 
\[
\left| {\bigwedge}_{\S}A_i \right| \ge \min \left\{p, n|A|-n^2+1\right\},
\]
thus establishing a long-standing conjecture of Erd\H os and Heilbronn
\cite{EG}.
Their proof exploited the properties of cyclic spaces of derivations on
exterior product spaces and the representation theory of symmetric groups;
see \cite{ANR} for another proof based on the polynomial method.
A far reaching generalization was obtained by Hou and Sun \cite{HS}. 
Here we use Lemma \ref{interpol} to reestablish their result in a short and
elegant manner, thereby also providing a simplified proof to the
Dias da Silva--Hamidoune theorem. Note that although our formulation below is
slightly different, it is still equivalent to \cite[Theorem 1.1]{HS}.

\begin{theorem}
Let $A_1,\dots,A_n$ be subsets of a field $\F$ such that $|A_i|=k$
for $1\le i\le n$ and assume that $S_{ij}\subseteq \F$ satisfy $|S_{ij}|\le s$
for $1\le i<j\le n$. If either $\char(\F)=0$ or 
\[
\char(\F)>\max\left\{n\lceil s/2\rceil,n(k-1)-n(n-1) \lceil s/2\rceil\right\},
\]
then
\[
\left| {\bigwedge}_{\S}A_i \right| \ge n(k-1)-n(n-1) \lceil s/2\rceil+1.
\]
\label{hou-sun}
\end{theorem}

\begin{proof}
Since posing extra restrictions cannot increase the size of the sumset,
we will assume that $s$ is even and $|S_{ij}|=s=2t$ holds for every pair $i<j$.
We may also assume that $k-1\ge (n-1)t$. We proceed by
way of contradiction. Suppose that $\bigwedge_\S A_i$ is contained in a  
set $C$ of size $n(k-1)-n(n-1)t$, and consider the polynomial
\[
\prod_{e\in C}(x_1+\dots+x_n-e)\times \prod_{i<j}
\left(\prod_{e\in S_{ij}}(x_j-x_i-e)\right) .
\]
This polynomial of degree $n(k-1)$ vanishes on the Cartesian product
$A_1\times\dots\times A_n$. According to Lemma \ref{interpol}, the coefficient
of the monomial $\prod x_i^{k-1}$ must be zero. This coefficient remains the
same if we slightly modify the polynomial and consider
\[
{F}(\bx)=\prod_{e=\binom{n}{2}t+1}^{n(k-1)-\binom{n}{2}t}(x_1+\dots+x_n-e)\times 
\prod_{i<j} \left( \prod_{e=-t}^{t-1}(x_j-x_i-e) \right)
\]
instead, keeping all leading terms intact. This coefficient is easy to
compute when one applies Lemma \ref{interpol} with 
$C_i\equiv\{0,1,\dots,k-1\}$. Indeed, if ${F}(\bc)\ne 0$ for some 
$\bc\in C_1\times\dots\times C_n$, then $|c_j-c_i|\ge t$ for every pair $i<j$.
Accordingly,
\[
\binom{n}{2}t\le c_1+\dots+c_n\le n(k-1)-\binom{n}{2}t,
\]
thus it must be $c_1+\dots+c_n=\binom{n}{2}t$ and the numbers $c_1,\dots,c_n$,
in some order, must coincide with the numbers $0,t,2t,\dots,(n-1)t$.
Moreover, it must be the natural order, for if $c_i>c_j$ for some $i<j$,
then $c_i-c_j\ge t+1$. Thus the computation of the coefficient reduces
to the evaluation of 
\[
\frac{{F}(c_1,c_2,\dots,c_n)}{\phi_1'(c_1)\phi_2'(c_2)\cdots\phi_m'(c_n)}
\]
at the point $\bc=\left(0,t,2t,\dots,(n-1)t\right)$.
After some cancellations this leads to the value
\[
(-1)^{\binom{n}{2}t}\times \frac{\big( n(k-1)-n(n-1)t\big)!}{(t!)^n}
\times\prod_{i=1}^n\frac{(it)!}{\big(k-1-(i-1)t\big)!}
\]
which is not zero in view of the assumption on the characteristic of the field.
This contradiction completes the proof.
\end{proof}

The tightness of the bound is demonstrated by the choice 
$$
A_i\equiv \{0,1,\dots,k-1\},\  S_{ij}\equiv \{-t+1,-t+2,\dots,t-1\}.
$$

\subsection{Further examples}

The alert reader must have already extracted from the above argument
the following general statement about restricted sumsets, which is
rather folklore, cf. \cite[Theorem 2.1]{ANR}.

\begin{theorem}
Let $d_i,s_{ij}$ denote non-negative integers, and
let $A_1,\dots,A_n$ and $S_{ij}$ ($1\le i<j\le n$) be subsets of a field $\F$
with $|A_i|=d_i+1$, $|S_{ij}|=s_{ij}$. Assume that $N=\sum d_i-\sum s_{ij}\ge 0$.
If the coefficient of the monomial $\prod x_i^{d_i}$ in the polynomial
\[
F_0(\bx)=\left( x_1+\dots+x_n \right)^N \prod_{i<j} (x_j-x_i)^{s_{ij}} \in
\F[x_1,\dots,x_n]
\]
is non-zero, then $\left| {\bigwedge}_{\S}A_i \right|> N$.
\end{theorem}

In the proof of Theorem \ref{hou-sun} we applied Lemma \ref{interpol}
in the case $d_i\equiv k-1$, $s_{ij}\equiv 2t$ to obtain the coefficient
$[x_1^{d_1}\dots x_n^{d_n}]F_0$
in a simple product form. Similar arguments work in the following cases.
The first example concerns a related result of Sun and Yeh,
cf. \cite[Theorem 1.1]{SY}, which involves only a minor modification.

\begin{example}
Let $d_i=k-i$, $s_{ij}\equiv 2t-1$. Then $N=n(k-1)-n(n-1)t$ and
\[
[x_1^{d_1}\dots x_n^{d_n}]F_0=
(-1)^{\binom{n}{2}t}\times \frac{N!}{(t!)^nn!}
\times\prod_{i=1}^n\frac{(it)!}{\big(k-1-(i-1)t\big)!}.
\]
\end{example}

\begin{proof}
Apply Lemma \ref{interpol} to the modified polynomial
\[
{F}(\bx)=\prod_{e=\binom{n}{2}t+1}^{n(k-1)-\binom{n}{2}t}(x_1+\dots+x_n-e)\times 
\prod_{i<j} \left( \prod_{e=1-t}^{t-1}(x_j-x_i-e) \right)
\]
with the choice $C_i=\{0,1,\ldots,k-1\}\setminus \{jm\mid 0\le j<i-1\}$.
Once again $F(\bc)\ne 0$ for $\bc\in C_1\times\dots\times C_n$ if and only if
$c_i=(i-1)t$ for every $1\le i\le n$, and the slight changes in the computation
are easy to detect.
\end{proof}

The next example considers the Alon--Nathanson--Ruzsa theorem \cite{ANR}.
Although our approach is not significantly different from the original proof,
we include it for it represents an atypical application of Lemma 
\ref{interpol}, when more than one $\bc\in C_1\times\dots\times C_n$
contributes to a non-zero summand.

\begin{example}
Let the $d_i$ be arbitrary, $s_{ij}\equiv 1$.
Then $N=d_1+\dots+d_n-\binom{n}{2}$ and
\[
[x_1^{d_1}\dots x_n^{d_n}]F_0=\frac{N!}{d_1!\dots d_n!}\prod_{i<j}(d_j-d_i).
\]
\end{example}

\begin{proof}
Replace the polynomial $F_0$ by
\[
{F}(\bx)=\prod_{e=\binom{n}{2}+1}^{d_1+\dots+d_n}(x_1+\dots+x_n-e)\times 
\prod_{i<j} (x_j-x_i) 
\]
and apply Lemma \ref{interpol} with $C_i=\{0,1,\ldots,d_i\}$.
Consider an element $\bc\in C_1\times\dots\times C_n$ whose coordinates
$c_i$ are mutually different.
Then $F(\bc)\ne0$ only if $\{c_1,\dots,c_n\}=\{0,\dots,n-1\}$.
That is, there is a permutation $\pi=\pi_{\bc}\in \Symm(n)$
such that $c_i=\pi_{\bc}(i)-1$. For such a $\bc$,
\[
F(\bc)=(-1)^NN!\times  \sign(\pi_{\bc}){\prod}_{i<j}(j-i),\qquad 
\phi'(c_i)=(-1)^{d_i-c_i}c_i!(d_i-c_i)!.
\]
Since $\binom{d_i}{c_i}=0$ for 
$d_i<c_i$, it is enough to prove that
\[
\sum_{\pi\in \Symm(n)}\sign(\pi)\prod_{i=1}^n\binom{d_i}{\pi(i)-1}=
\prod_{i<j}\frac{d_j-d_i}{j-i}.
\]
To establish this identity, notice that both sides are completely 
antisymmetric polynomials of minimum possible 
degree $n(n-1)/2$ in the variables $d_i$, which attain the same value at 
$(d_1,\dots,d_n)=(0,\dots,n-1)$. 
\end{proof}

\noindent{\em Remark.} A more direct proof goes as follows.
Write $x^{[k]}=x(x-1)\dots(x-k+1)$ and consider the polynomials
\[
F(\bx)=\left( \sum_{i=1}^n x_i-\binom{n}{2}\right)^{[N]}
\prod_{i<j}(x_j-x_i),\quad
F^*(\bx)=\sum_{k_1+\dots+k_n=d_1+\dots+d_n}\frac{N!}{\prod k_i!}
\prod_{i<j}(k_j-k_i)\prod_{i=1}^n x_i^{[k_i]}.
\]
It is enough to prove that $F-F^*$ vanishes on the Cartesian product of the 
sets $C_i=\{0,1,\dots,d_i\}$, for then 
$[x_1^{d_1}\dots x_n^{d_n}](F-F^*)=0$ by Lemma \ref{interpol} and therefore
\[
[x_1^{d_1}\dots x_n^{d_n}]F_0=[x_1^{d_1}\dots x_n^{d_n}]F=
[x_1^{d_1}\dots x_n^{d_n}]F^*=\frac{N!}{d_1!\dots d_n!}\prod_{i<j}(d_j-d_i)
\]
as claimed. For the proof, notice that $\bc\in C_1\times\dots\times C_n$
implies $F(\bc)=F^*(\bc)=0$ unless $c_i=d_i$ for every $i$, in which
case $F(\bc)=F^*(\bc)$ follows from the very choice of the coefficients
in $F^*$. This argument can be extended to show that in fact $F=F^*$.

\medskip
Our final example originates in Xin \cite{Xin}, where it appears in 
the form of the constant term identity
\begin{equation}
\CT\left[ x_1^{-a_1}\dots x_n^{-a_n}(x_1+\dots+x_n)^{a_1+\dots+a_n}
{\prod}_{i\ne j}(1-x_j/x_i)^{a_i}     \right]=\binom{\abs{\ba}}{\ba},
\label{xin}
\end{equation}
see also \cite{GLXZ}. Here the full capacity of 
Theorem \ref{H-interpol} can be exploited with a minimum amount of
computation.

\begin{example}
Let $d_i=na_i$, $s_{ij}=a_i+a_j$. Then $N=a_1+\dots+a_n$ and 
\[
[x_1^{d_1}\dots x_n^{d_n}]F_0= (-1)^{\sum_{i<j}a_i}\binom{\abs{\ba}}{\ba}.
\]
\end{example}

\begin{proof}
For the proof we may assume that $\char(\F)=0$. Choose an arbitrary set
$B=\{b_1,\dots,b_n\}\subset \F$ so that $b_1+\dots+b_n=0$, and consider the
multisets $C_1,\dots,C_n$ with $\supp(C_i)=B$ and multiplicity functions
given by $\omega_i(b_j)=a_i+\chi(j=i)$; then $|C_i|=d_i+1$. We apply
Theorem \ref{H-interpol} to the polynomial
\[
F(\bx)=(-1)^{\sum_{i<j}a_i}F_0(\bx)=
(x_1+\dots+x_n)^{a_1+\dots+a_n} {\prod}_{i\ne j}(x_i-x_j)^{a_i}. 
\]
There is only one non-zero summand in the summation formula
for $[x_1^{d_1}\dots x_n^{d_n}]F$.
Indeed, suppose that
\[
\frac{\partial^{m_1+\dots+m_n}F}{\partial x_1^{m_1}\dots
\partial x_n^{m_n}}(c_1,\ldots,c_n)\ne 0
\]
for some $\bc\in C_1\times\dots\times C_n$ with $0\le m_i<\omega_i(c_i)$.
First we show that the coordinates $c_i$ are mutually different.
Assume that on the contrary, $c_i=c_j$ for some $i\ne j$.
Then $m_i+m_j\le \omega_i(c_i)+\omega_j(c_i)\le a_i+a_j-1$.
This implies that the polynomial $H:=\prod(\partial/\partial x_i)^{m_i}F$ 
is divisible by $x_i-x_j$, a contradiction.

Thus, $\{c_1,\dots,c_n\}=\{b_1,\dots,b_n\}$. Note that $m_i\le a_i$. 
If $\sum m_i<\sum a_i$, then $H$ is divisible by $\sum x_i$, a contradiction.
Accordingly, $m_i=a_i$, $c_i=b_i$ for every $i$. Moreover, all the 
$a_1+\dots+a_n$ partial derivatives must be applied to the term
$(x_1+\dots+x_n)^{a_1+\dots+a_n}$ in $F$. After all, we get
\[
[x_1^{d_1}\dots x_n^{d_n}]F=\prod_{i=1}^n\kappa(C_i,b_i,a_i)
\frac{\partial^{a_1+\dots+a_n}F}{\partial x_1^{a_1}\dots
\partial x_n^{a_n}}(b_1,\ldots,b_n)=\binom{\abs{\ba}}{\ba},
\]
for $\prod_{i\ne j}(b_i-b_j)^{a_i}=\prod_i\prod_{c\in B\setminus\{b_i\}}(b_i-c)^{a_i}$.
\end{proof}

\noindent
{\em Remarks.}
{\em 1.} A connection between restricted sumsets and
Morris's  constant term identity was made recently by Zhou \cite{Zhou2}.
{\em 2.} Let $h_r(\bx)=\sum_{1\le j_1\le \dots\le j_r\le n}x_{j_1}\dots x_{j_r}$
denote the complete symmetric function of degree $r$. Following
Good's method \cite{Good} one gets the following generalization of
(\ref{xin}), also implicit in \cite{Xin}:
\[
\CT\left[ x_1^{-ra_1}\dots x_n^{-ra_n}h_r(x_1,\dots,x_n)^{a_1+\dots+a_n}
{\prod}_{i\ne j}(1-x_j/x_i)^{a_i}     \right]=\binom{\abs{\ba}}{\ba}.
\]
It would be interesting to obtain a proof of this identity based
on the Combinatorial Nullstellensatz.

\bigskip
\section{On a problem of Kadell}
\label{4}

The aforementioned idea of Aomoto
led Kadell \cite{Kadell2} to discover and prove
the following Dyson-type identity.
Fix $m<n$. For $1\le r\le n$ and an $m$-element subset $M$ of 
$\{1,2,\dots,n\}$, consider the Laurent polynomial
\[
\K_{r,M}(\bx;\ba)=
\bigg(1+\sum_{v\not\in M}a_v\bigg)\prod_{s\in M}\left(1-\frac{x_r}{x_s}\right)
\D(\bx;\ba).
\]
Note that $\K_{r,M}(\bx;\ba)=0$ if $r\in M$. Then, according to 
\cite[Theorem 1]{Kadell2},
\begin{equation}
\CT\Bigg[\sum_{r=1}^n\sum_{\abs{M}=m}\K_{r,M}(\bx;\ba)\Bigg]=
n\binom{n-1}{m}(1+\abs{\ba})\binom{\abs{\ba}}{\ba}.
\label{kadellres}
\end{equation}
Kadell suggested that each non-zero function $\K_{r,M}(\bx;\ba)$ must have
the same contribution to the constant term. He formulated an even more
general hypothesis (see \cite[Conjecture 2]{Kadell2}), which was established
recently by Zhou \cite{Zhou} based on the first layer formulas for 
Dyson-coefficients \cite[Theorem 1.7]{LXZ}. Kadell also suggested the
following $q$-analogue of his hypothesis.

\begin{conjecture}[{\cite[Conjecture 3]{Kadell2}}]
Let $M\subset \{1,2,\dots,n\}$ and $\{r_s\mid s\in M\}\cap M=
\emptyset$. Then
\[
\CT\Bigg[\prod_{1\leq i < j \leq n} 
\left(\frac{x_i}{x_j}\right)_{a_i^*}
\left(\frac{qx_j}{x_i}\right)_{a_j^*}\Bigg]=
\frac{1-q^{1+\abs{\ba}}}{1-q^{1+\sum_{v\not\in M}a_v}}\qbin{\abs{\ba}}{\ba},
\]
where, with a slight abuse of notation, 
$a_i^*=a_i^*(j)=a_i+\chi(j\in M,i=r_j)$
and $a_j^*=a_j^*(i)=a_j+\chi(i\in M,j=r_i)$.
\label{kadellconj}
\end{conjecture}

\noindent
Zhou \cite{Zhou} pointed out that this conjecture already fails 
for $n=3$, $\abs{M}=m=1$, and proved
a meaningful $q$-analogue of \cite[Conjecture 2]{Kadell2}, which is too
technical to be recalled here in detail. Our main result 
in this section is the proof of
the following special case of Conjecture \ref{kadellconj}
corresponding to $M=\{1,\dots,m\}$ and $r_s\equiv n$ that, unexpectedly,
does not seem to be implied by Zhou's result.

\begin{theorem}
Let $m<n$. Then
\[
\CT\Bigg[\prod_{1\leq i < j \leq n} 
\left(\frac{x_i}{x_j}\right)_{a_i}
\left(\frac{qx_j}{x_i}\right)_{a_j^*}\Bigg]=
\frac{1-q^{1+\abs{\ba}}}{1-q^{1+\sum_{v=m+1}^na_v}}\qbin{\abs{\ba}}{\ba},
\]
where $a_n^*=a_n+\chi(i\le m)$ and $a_j^*=a_j$ otherwise.
\label{main}
\end{theorem}

Specializing at $q=1$ and taking into account the symmetry of the
Dyson product we obtain the following special case of 
\cite[Conjecture 2]{Kadell2}, which already implies (\ref{kadellres}).

\begin{corollary}
Let $m<n$, $M\subset \{1,\dots,n\}$ with $\abs{M}=m$ and 
$r\in\{1,\dots,n\}\setminus M$. Then
\[
\CT\Bigg[\prod_{s\in M}\left(1-\frac{x_r}{x_s}\right)\D(\bx;\ba)\Bigg]=
\frac{1+\abs{\ba}}{1+\sum_{v\not\in M}a_v}\binom{\abs{\ba}}{\ba}.
\]
\end{corollary}

As a final remark we mention that the $m=1$ case of this corollary
in conjunction with the Zeilberger--Bressoud theorem
immediately implies Sills' \cite[Theorem 1.1]{Sills}: For
$1\le r\ne s\le n$,
\[
\CT\big[(x_r/x_s)\D(\bx;\ba)\big]=
\frac{-a_s}{1+\abs{\ba}-a_s}\binom{\abs{\ba}}{\ba}.
\]
In general, one may use the inclusion-exclusion principle to obtain a
formula for the constant term of
\[
\left({x_r^m}/{{\prod}_{s\in M}x_s}\right)\D(\bx;\ba),
\]
in agreement with \cite[Theorem 1.7]{LXZ}.

\bigskip
\noindent{\em Proof of Theorem \ref{main}.}
Note that if $a_i=0$ for some $i<n$, then we may omit all factors that include 
the variable $x_i$ without affecting the constant term.
Accordingly, we may assume that each $a_i$, with the possible exception
of $a_n$, is a positive integer. Clearly the constant term equals the
coefficient of the monomial
\[
\prod_{i=1}^n x_i^{\abs{\ba}-a_i+\chi(i\le m)}
\]
in the homogeneous polynomial
\[
F(\bx)=\prod_{1\le i<j\le n}
\Bigg(\prod_{k=0}^{a_i-1}{\big(x_j-x_iq^k\big)}\times
\prod_{k=1}^{a_j^*}{\big(x_i-x_jq^k\big)} \Bigg),
\]
where 
\[
a_j^*=\begin{cases}\displaystyle
a_j+1 & \text{if $j=n$ and $i\le m$,} \\[3mm]
a_j & \text{otherwise.}
\end{cases}
\]
To express this coefficient we apply Lemma \ref{interpol} with
$\mathbb{F}=\mathbb{Q}(q)$. Once again, the aim is to choose 
the sets $C_i$ so that $F({\bc})=0$ for all but
one element ${\bc}\in C_1\times\dots\times C_n$. This can be easily
achieved as follows. Let $C_i=\{q^{\alpha_i}\mid \alpha_i\in B_i\}$, where
\[
B_i=\{0,1,\ldots,\abs{\ba}-a_i+\chi(i\le m)\}.
\]
The sets $C_i$ clearly have the right cardinalities. 
Now assume that $c_i=q^{\alpha_i}\in C_i$ and $F(\bc)\ne 0$.
Then all the $\alpha_i$ are distinct. Moreover, 
\[
\alpha_j\ge \alpha_i+a_i+\chi(j<i)+\chi(i=n, j\le m)
\]
holds for $\alpha_j> \alpha_i$. Next consider the unique permutation 
$\pi\in \Symm_n$ for which 
\[
0\le \alpha_{\pi(1)}< \alpha_{\pi(2)}<\dots< \alpha_{\pi(n)}\le 
\abs{\ba}-a_{\pi(n)}+\chi\big(\pi(n)\le m\big).
\]
We obtain the chain of inequalities
\[
\abs{\ba}-a_{\pi(n)}=
\sum_{i=1}^{n-1}a_{\pi(i)}\le 
\sum_{i=1}^{n-1}\big(\alpha_{\pi(i+1)}-\alpha_{\pi(i)}\big)=
\alpha_{\pi(n)}-\alpha_{\pi(1)}\le \abs{\ba}-a_{\pi(n)}+1.
\]
Notice that the first inequality is strict if $\pi$ is not the identity
permutation, while the second inequality is strict if $\pi(n)>m$.
Suppose that $\pi(n)\ne n$. Since $\pi\ne\id$, it must be $\pi(n)\le m$.
Consider the index $i$ with $\pi(i)=n$. Then 
$\alpha_{\pi(i+1)}-\alpha_{\pi(i)}= a_{\pi(i)}+1$,
which implies $\pi(i+1)>m$. Therefore there must be an index $i+1\le j<n$
such that $\pi(j)>m$ and $\pi(j+1)\le m$. For such a $j$ we have
$\alpha_{\pi(j+1)}-\alpha_{\pi(j)}\ge  a_{\pi(j)}+1$, resulting in
\[
\abs{\ba}-a_{\pi(n)}=
\sum_{i=1}^{n-1}a_{\pi(i)}\le 
\sum_{i=1}^{n-1}\big(\alpha_{\pi(i+1)}-\alpha_{\pi(i)}\big)-2
\le \abs{\ba}-a_{\pi(n)}-1, 
\]
a contradiction.
Thus we can conclude that $\pi(n)=n$, implying $\pi= \id$ and 
$\alpha_i=a_1+\dots +a_{i-1}$ for every $i$. 

It only remains to substitute these values into
\[
\frac{F(c_1,c_2,\dots,c_n)}{\phi_1'(c_1)\phi_2'(c_2)\cdots\phi_n'(c_n)},
\]
which is quite a routine calculation. Therefore we only recall that
substituting the same values in the same formula working with
\[
F(\bx)=\prod_{1\le i<j\le n}
\Bigg(\prod_{k=0}^{a_i-1}{\big(x_j-x_iq^k\big)}\times
\prod_{k=1}^{a_j}{\big(x_i-x_jq^k\big)} \Bigg)
\]
and $B_i=\{0,1,\ldots,\abs{\ba}-a_i\}$ yields the $q$-Dyson
constant term $\CT[\D_q(\bx;\ba)]$, see \cite{KN}. 
The changes are easily detected, and noting
$\alpha_i+a_i=\alpha_{i+1}$, $\alpha_n+a_n=\abs{\ba}$ we find
that the constant term in question is indeed
\[
\prod_{i=1}^m\frac{q^{\alpha_i}-q^{\alpha_n+a_n+1}}{q^{\alpha_i}-q^{\abs{a}-a_i+1}} 
\qbin{\An}{\ba}=\frac{1-q^{1+\abs{\ba}-\alpha_1}}{1-q^{1+\abs{\ba}-\alpha_{m+1}}}
\qbin{\An}{\ba},
\]
as claimed.
\qed

\bigskip
\section{A new proof of the Selberg integral}
\label{5}

Due to its equivalence to the Selberg integral, it will be enough to 
establish Morris's constant term identity (\ref{morris}). Making the
Laurent polynomial homogeneous by the introduction of a new variable does not
affect the constant term. Thus, we are to determine the constant term of the
Laurent polynomial
\[
\M(x_0,\bx;a,b,k):= \prod_{j=1}^n \left(1-\frac{x_j}{x_0}\right)^a
\left(1-\frac{x_0}{x_j}\right)^b
\prod_{1\leq i \neq j \leq n}\left(1-\frac{x_i}{x_j}\right)^k,
\] 
which is the same as the coefficient of 
$x_0^{na}\prod_{i=1}^nx_i^{(n-1)k+b}$
in the homogeneous polynomial
\[
\prod_{j=1}^n(x_0-x_j)^a(x_j-x_0)^b\times \prod_{1\leq i \neq j \leq n}(x_j-x_i)^k.
\]
As in Section 3, we modify this polynomial without affecting this leading
coefficient and consider
\begin{equation}
F(x_0,\bx)=\prod_{j=1}^n\prod_{e=-a}^{b-1}(x_j-x_0-e)\times
\prod_{1\le i<j\le n}\prod_{e=-k}^{k-1}(x_j-x_i-e).
\label{m-poly}
\end{equation}
To apply Lemma \ref{H-interpol} efficiently, we choose sets
$C_i=\{0,1,\dots,(n-1)k+b\}$ for $1\le i\le n$ and multiset
\[
C_0=\{0\}\cup \bigcup_{\ell=0}^{n-1}\{k\ell+1,k\ell+2,\dots,k\ell+a\}.
\]
They have the right cardinality, the latter one being an ordinary set 
if $k\ge a$. Consider $c_i\in {C_i}$ and $m_i<\omega_i(c_i)$.
Note that $m_i=0$ for $1\le i\le n$. We proceed to prove that 
\[
\frac{\partial^{m_0+\dots+m_n}F}{\partial x_0^{m_0}\dots
\partial x_n^{m_n}}(c_0,\ldots,c_n)= 
\frac{\partial^{m_0}F}{\partial x_0^{m_0}}(c_0,\ldots,c_n)=0
\]
for all but one such selection.

\begin{lemma}
If $c_0\ne 0$, then
\begin{equation}
\frac{\partial^{m_0}F}{\partial x_0^{m_0}}(c_0,\ldots,c_n)=0.
\label{nulla}
\end{equation}
\end{lemma}

\begin{proof}
Write $S_\ell=\{k\ell+1,k\ell+2,\dots,k\ell+a\}$. Since $m_0<\omega_0(c_0)$, 
there is an index $0\le u\le n-\omega_0(c_0)$ 
such that $c_0\in S_u\cap S_{u+1}\cap \dots \cap S_{u+m_0}$. That is,
\[
(u+m_0)k+1\le c_0\le uk+a.
\]
Accordingly, if $c_j$ lies in the interval $[uk, (u+m_0)k+b]$
for some $1\le i\le n$, then
$c_0-a\le c_j\le c_0+b-1$ and there is a term of the form $x_j-x_0-e$ in $F$
which attains 0 when evaluated at the point $(c_0,\bc)$.
It follows, that (\ref{nulla}) holds if more than $m_0$ of such $c_i$ 
lie in the interval $[uk, (u+m_0)k+b]$.

Otherwise either at least $u+1$ of $c_1,\dots,c_n$ lie in the interval
$[0,ku-1]$, or at least $n-m_0-u$ of them lie in the interval
$[(u+m_0)k+b+1,(n-1)k+b]$. In either case there is a pair of indices 
$1\le i<j\le n$ such that $|c_j-c_i|<k$, meaning that there is a term of the 
form $x_j-x_i-e$ in $F$ which attains 0 when evaluated at the point 
$(c_0,\bc)$, and once again we arrive at (\ref{nulla}).
\end{proof}

Thus we only have to consider the case when $c_0=0$; then $\omega_0(c)=1$ and
$m_0=0$. If
\[
\frac{\partial^{m_0}F}{\partial x_0^{m_0}}(c_0,\ldots,c_n)=F(c_0,\bc)\ne 0,
\]
then $c_1,\dots,c_n\in [b,(n-1)k+b]$ and $|c_j-c_i|\ge k$ for each pair
$1\le i<j\le n$. Therefore the numbers $c_1,\dots,c_n$,
in some order, must coincide with the numbers $b,k+b,\dots,(n-1)k+b$.
Moreover it must be the natural order, for if $c_i>c_j$ for some $i<j$,
then $c_i-c_j\ge k+1$. It only remains to evaluate
\[
\prod_{i=0}^n \kappa(C_i,c_i,0) F(c_0,\bc)
\]
at the point $(c_0,\bc)=(0,b,k+b,\dots,(n-1)k+b)$.
Since $\omega_i(c_i)=0$ for each $i$, we simply have 
\[
\kappa(C_i,c_i,0)=
\frac{1}{\prod_{c \in {C_i} \setminus\{c_i\}}(c_i-c)^{\omega_i(c)}}
\]
and one easily recovers (\ref{morris}).
\medskip

\noindent
{\em Remark.}
For the sake of simplicity,
we tacitly assumed that the parameters $a,b,k$ are positive integers. 
It is not difficult to modify the above proof to suit the remaining
cases and we leave it to the reader. Alternatively, one can easily
reduce the $k=0$ case to the Chu-Vandermonde identity, whereas the
$\min\{a,b\}=0$ case is just the equal parameter case of Dyson's
identity.

\bigskip
\section{Interlude}
\label{6}

Replace the polynomial in (\ref{m-poly}) by
\[
\prod_{j=1}^n\prod_{e=1}^{a}(x_0-q^ex_j)\prod_{e=0}^{b-1}(x_j-q^ex_0)\times
\prod_{1\le i<j\le n}\prod_{e=0}^{k-1}(x_j-q^ex_i)\prod_{e=1}^{k}(x_i-q^ex_j).
\]
Also replace the multisets $C_i$ by multisets which consist of powers of $q$
whose exponents belong to $C_i$, and with the same multiplicities. Repeating
the proof given in the previous section almost verbatim
one obtains without any difficulty the following
version of the $q$-Morris constant term identity:
\begin{equation*}
\CT\Bigg[ \prod_{j=1}^n(qx_j)_a(1/x_j)_b \D_q(\bx;k) \Bigg]=
\prod_{j=0}^{n-1}\frac{(q)_{a+b+kj}(q)_{kj+k}}{(q)_{a+kj}(q)_{b+kj}(q)_{k}}.
\label{q-morris}
\end{equation*}
Although the identity conjectured in Morris's thesis \cite{Morris} reads
slightly differently as
\begin{equation}
\CT\Bigg[ \prod_{j=1}^n(x_j)_a(q/x_j)_b \D_q(\bx;k) \Bigg]=
\prod_{j=0}^{n-1}\frac{(q)_{a+b+kj}(q)_{kj+k}}{(q)_{a+kj}(q)_{b+kj}(q)_{k}},
\label{q-morris'}
\end{equation}
the two are easily seen to be equivalent, for each monomial of degree zero
has the same coefficient in the Laurent polynomials
$\prod_{j=1}^n(qx_j)_a(1/x_j)_b$ and $\prod_{j=1}^n(x_j)_a(q/x_j)_b$.
Morris's conjecture was established independently in \cite{Habsieger} and
\cite{Kadell1.5} via the proof of a $q$-Selberg integral proposed by Askey
\cite{Askey}, followed by a  more elementary proof in \cite{Zeilberger}.
\medskip

The above argument relates to the one given in the previous section 
in a similar way as
the derivation of the $q$-analogue of Dyson's conjecture
in \cite{KN} relates to the original version of Karasev and Petrov's proof
\cite{KP} for the Dyson product.
One may say that applications of Lemma \ref{interpol} (or its 
generalization Theorem \ref{H-interpol}) allows one to prove an appropriate 
$q$-analogue practically along the same lines as the original identity,
even without the need to modify the corresponding polynomial.
This works also the
other way around: the way (\ref{q-morris'}) is formulated gives a hint
of an alternative proof of (\ref{morris}) which involves a slightly
different modification along with a slightly different choice of the 
multisets $C_i$. Our preference was given to the modification, which 
allowed a more simple choice for the $C_i$ as well as to keep the
natural order of the variables $x_0,x_1,\dots,x_n$ for the $q$-analogue
in the following sense. 

All the constant term identities and their $q$-analogues studied in this
paper can be formulated in the following context. Let $\bB=(\beta_{ij})$ 
denote an $(n+1)\times(n+1)$ matrix with rows and columns numbered from
0 to $n$, corresponding to the natural order of the variables. It is assumed
that the entries are non-negative integers and
all the diagonal entries are zero. Associated to such a matrix is
the Laurent polynomial
\[
\L(x_0,\bx;\bB)=\prod_{0\leq i \neq j \leq n}\left(1-\frac{x_i}{x_j}\right)^{\beta_{ij}}
\]
and its $q$-analogue
\[
\L_q(x_0,\bx;\bB)=\prod_{0\leq i < j \leq n} 
\left(\frac{x_i}{x_j}\right)_{\beta_{ij}}\left(\frac{qx_j}{x_i}\right)_{\beta_{ji}}.
\]
Thus, one can write 
$\D(\bx;\ba)=\L(x_0,\bx;\bB_\D)$ and $\M(x_0,\bx;a,b,k)=\L(x_0,\bx;\bB_\M)$
with the matrices
\[
\bB_\D=
\left(
\begin{array}{c | ccccc}
0 & 0& 0 &0 & \ldots & 0\\ \hline
0 & 0 & a_1 &  a_1 & \ldots & a_1 \\
0 & a_2 & 0 & a_2  & \ldots & a_2 \\
0 & a_3 & a_3 & 0  & \ldots & a_3 \\
\vdots  & \vdots  & \vdots  & \vdots  & \ddots  & \vdots  \\
0 & a_n & a_n & a_n & \ldots & 0 \\
\end{array}
\right)
\qquad \textrm{and} \qquad
\bB_\M=
\left(
\begin{array}{c | ccccc}
0 & b& b &b & \ldots & b\\ \hline
a & 0 & k &  k & \ldots & k \\
a & k & 0 & k  & \ldots & k \\
a & k & k & 0  & \ldots & k \\
\vdots  & \vdots  & \vdots  & \vdots  & \ddots  & \vdots  \\
a & k & k & k & \ldots & 0 \\
\end{array}
\right)
\]
corresponding to the Dyson resp. Morris constant term identities, 
whereas $\D_q(\bx;\ba)=\L_q(x_0,\bx;\bB_\D)$. Note that simultaneous permutation
of the rows and columns of $\bB$ according to the same element of 
$\Symm_{n+1}$ has no effect on $\CT[\L(x_0,\bx;\bB)]$. Generally it is not the
case for $\CT[\L_q(x_0,\bx;\bB)]$, but as we explained in relation to
the $q$-Morris identity, one may always apply the cyclic permutation
\[
n\to n-1\to \dots\to 1\to 0\to n
\]
or any of its powers 
without affecting the constant term.

Theorem \ref{main} concerns $\CT[\L_q(x_0,\bx;\bB_\K)]$ for the matrix 
\[
\bB_\K=
\left(
\begin{array}{c|cccc|cccc}
0 & 0 & 0 & \ldots & 0 & 0 & \ldots & 0 & 0 \\ \hline
0 & 0 & a_1 & \ldots & a_1 & a_1 & \ldots & a_1 & a_1 \\
0 & a_2 & 0 & \ldots & a_2 & a_2 & \ldots & a_2 & a_2 \\
\vdots & \vdots & \vdots  & \ddots  & \vdots  
& \vdots & \ddots  & \vdots  & \vdots\\
0 & a_m & a_m & \ldots & 0 & a_m & \ldots & a_m & a_m \\ \hline
0 & a_{m+1} & a_{m+1} & \ldots & a_{m+1} & 0 & \ldots & a_{m+1} & a_{m+1} \\
\vdots & \vdots & \vdots  & \ddots  & \vdots  
& \vdots & \ddots  & \vdots  & \vdots\\
0 & a_{n-1} & a_{n-1} & \ldots & a_{n-1} & a_{n-1} & \ldots & 0 & a_{n-1} \\
0 & a_n+1 & a_n+1 & \ldots & a_n+1 & a_n & \ldots & a_n & 0 \\
\end{array}
\right).
\]
Applying the above mentioned
cyclic permutations to $\bB_\K$, after rearranging indices  
we obtain the following more general form of Theorem \ref{main}.

\begin{theorem}
Fix an arbitrary integer $r\in \{1,2,\dots,n\}$. Then 
Conjecture \ref{kadellconj}
is valid with the choice of $M=\{r+1,\dots,r+m\}$ and $r_s\equiv r$, where
indices are understood modulo $n$.
\end{theorem}

Aomoto's identity (\ref{aomoto}) and Forrester's conjecture are related
to the matrices 
\[
\bB_\A=
\left(
\begin{array}{c | ccc|ccc}
0 & b & \ldots & b & b & \ldots & b \\ \hline
a & 0 & \ldots & k & k & \ldots & k \\
\vdots & \vdots & \ddots & \vdots  & \vdots & \ddots & \vdots \\
a & k & \ldots & 0 & k & \ldots & k \\ \hline
a+1 & k & \ldots & k & 0 & \ldots & k \\
\vdots & \vdots & \ddots & \vdots  & \vdots & \ddots & \vdots \\
a+1 & k & \ldots & k & k & \ldots & 0 \\
\end{array}
\right)
\qquad \textrm{and} \qquad
\bB_\Fo=
\left(
\begin{array}{c | ccc|ccc}
0 & b & \ldots & b & b & \ldots & b \\ \hline
a & 0 & \ldots & k & k & \ldots & k \\
\vdots & \vdots & \ddots & \vdots  & \vdots & \ddots & \vdots \\
a & k & \ldots & 0 & k & \ldots & k \\ \hline
a & k & \ldots & k & 0 & \ldots & k+1 \\
\vdots & \vdots & \ddots & \vdots  & \vdots & \ddots & \vdots \\
a & k & \ldots & k & k+1 & \ldots & 0 \\
\end{array}
\right),
\]
where the last $m$ resp. $n-n_0$ rows/columns are separated.
In the first case we rearranged the matrix so that a $q$-analogue can
be formulated within our framework. Our main result concerns the
overlay of these matrices when $m\ge n-n_0$, that is, the matrix
\[
\bB_{\A\Fo}=
\left(
\begin{array}{c | ccc|ccc|ccc}
0 & b & \ldots & b & b & \ldots & b & b & \ldots & b \\ \hline
a & 0 & \ldots & k & k & \ldots & k & k & \ldots & k \\
\vdots & \vdots & \ddots & \vdots  & \vdots & \ddots & \vdots 
& \vdots & \ddots & \vdots \\
a & k & \ldots & 0 & k & \ldots & k & k & \ldots & k \\ \hline
a+1 & k & \ldots & k & 0 & \ldots & k & k & \ldots & k \\
\vdots & \vdots & \ddots & \vdots  & \vdots & \ddots & \vdots 
& \vdots & \ddots & \vdots \\
a+1 & k & \ldots & k & k & \ldots & 0 & k & \ldots & k \\ \hline
a+1 & k & \ldots & k & k & \ldots & k & 0 & \ldots & k+1 \\
\vdots & \vdots & \ddots & \vdots  & \vdots & \ddots & \vdots & 
\vdots & \ddots & \vdots \\
a+1 & k & \ldots & k & k & \ldots & k & k+1 & \ldots & 0 \\
\end{array}
\right)\quad
\begin{array}{c}
0\\ \hline
1\\
\vdots\\
n-m\\ \hline
n-m+1\\
\vdots\\
n_0\\ \hline
n_0+1\\
\vdots\\
n\\ \end{array}\ .
\]

\begin{theorem}
Let $n$ be a positive integer. For arbitrary nonnegative integers
$a,b,k$ and $m,n_0\le n\le m+n_0$,
\[
\CT[\L_q(x_0,\bx;\bB_{\A\Fo})]= \prod_{j=0}^{n-1}
\frac{(q)_{a+b+kj+\chi(j>n_0)(j-n_0)+\chi(j\ge n-m)}(q)_{kj+\chi(j>n_0)(j-n_0)+k}}
{(q)_{a+kj+\chi(j>n_0)(j-n_0)+\chi(j\ge n-m)}(q)_{b+kj+\chi(j>n_0)(j-n_0)}(q)_k}\times
\prod_{j=1}^{n-n_0} \frac{1-q^{(k+1)j}}{1-q^{k+1}}.
\]
\label{a-f}
\end{theorem}

\noindent
When $m=0$, this proves Baker and Forrester's \cite[Conjecture 2.1]{BF},
and further specializing at $q=1$, Forrester's original conjecture as well.
The $n_0=n$ case gives the following $q$-analogue of Aomoto's identity.

\begin{corollary}
Let $n$ be a positive integer. For arbitrary nonnegative integers
$a,b,k$ and $m\le n$,
\[
\CT[\L_q(x_0,\bx;\bB_\A)]= \prod_{j=0}^{n-1}
\frac{(q)_{a+b+kj+\chi(j\ge n-m)}(q)_{kj+k}}{(q)_{a+kj+\chi(j\ge n-m)}(q)_{b+kj}(q)_k}.
\]
\end{corollary}

\noindent
A more general version of this identity, which involves an additional parameter
attached to $b$, was established in Kadell's paper \cite{Kadell1.5}. An
elementary proof was claimed recently by Xin and Zhou \cite{XZ}.
Replacing $k$ by $k+1$ we obtain that Theorem \ref{a-f} is valid for arbitrary
$m\le n$ when $n_0=0$. Although the condition $n\le m+n_0$ is crucial to our
proof given in the next section, it does not seem to be necessary.

\begin{conjecture}
Theorem \ref{a-f} remains valid without the restriction $n\le m+n_0$.
\end{conjecture}

\bigskip
\section{Proof of the conjecture of Forrester}
\label{7}

Clearly $\CT[\L_q(x_0,\bx;\bB)]$ equals the coefficient of
$\prod_j x_j^{B_j}$, where $B_j=\sum_i \beta_{ij}$, in the polynomial
\[
F_q(x_0,\bx;\bB):=\prod_{0\le i<j\le n} \Bigg(\prod_{t=0}^{\beta_{ij}-1}(x_j-q^tx_i)
\times \prod_{t=1}^{\beta_{ji}}(x_i-q^tx_j) \Bigg).
\]

\begin{claim}
Suppose that $c_i=q^{\alpha_i}$ for some integers $\alpha_i$ such that 
$F_q(c_0,\bc;\bB)\ne 0$. Let $j>i$. Then $\alpha_j\ge \alpha_i$ implies 
$\alpha_j\ge \alpha_i+\beta_{ij}$, and $\alpha_i>\alpha_j$ implies 
$\alpha_i\ge \alpha_j+\beta_{ji}+1$. Both statements are valid even if
the corresponding entry in $\bB$ is zero. The same is true with
$F_q$ replaced by any of its partial derivatives in which $m_i=m_j=0$. \qed
\label{trivi}
\end{claim}

\noindent
We are to apply Theorem \ref{H-interpol} with the polynomial
$F=F_q(.;\bB_{\A\Fo})$. 
As in Section 5, we will assume that the parameters $a,b,k$ are positive 
integers and leave the rest to the reader.

\subsection{The choice for the multisets $C_i$}
\label{Aichoice}

Write $\gamma_i=\beta_{in}$ for $0\le i<n$ and let $\De_t=\sum_{i=0}^t\gamma_i$.
Thus, 
\[
\gamma_0=b,\ \gamma_1=\dots=\gamma_{n_0}=k,\ 
\gamma_{n_0+1}=\dots=\gamma_{n-1}=k+1
\]
and $\beta_{ij}=\gamma_{\min\{i,j\}}$ for $1\le i\ne j\le n$.
Consider the intervals $I_t=[\De_t-\gamma_t+1,\De_t]=[\De_{t-1}+1,\De_t]$, where
here and thereafter $[u,v]$ stands for the set of integers $\ell$ satisfying
$u\le \ell\le v$. The intervals $I_0:=[0,b],I_1\dots,I_{n-1}$ 
are mutually disjoint.
The multisets $C_i$ are defined in the form
$C_i=\{q^{\alpha}\mid \alpha\in A_i\}$, where for $1\le j\le n$
\[
A_j=\{0\}\cup \bigcup_{t=0}^{n-1} 
\left[\De_{t}-\gamma_{\min\{t,j\}}+1,\De_t  \right]\subseteq 
\bigcup_{t=0}^{n-1}I_t=[0,\Delta_{n-1}]
\]
is and ordinary set and
\[
A_0= \{0\}\cup \bigcup_{t=0}^{n-1} 
\left[ \De_t-b+1, \De_t-b+ \beta_{t+1, 0}  \right] 
\]
is a multiset. 
Then $|C_i|=|A_i|=B_i+1$ holds for every $0\le i\le n$.
We are to show that 
\[
\frac{\partial^{m_0+\dots+m_n}F}{\partial x_0^{m_0}\dots
\partial x_n^{m_n}}(c_0,\ldots,c_n)= 
\frac{\partial^{m_0}F}{\partial x_0^{m_0}}(c_0,\ldots,c_n)=0
\]
for all but one selection of elements $c_i\in {C_i}$ and 
multiplicities $m_i<\omega_i(c_i)$, namely when $c_0=1$, $c_i=q^{\Delta_{i-1}}$
for $1\le i\le n$, and all the multiplicities are zero.

\subsection{The combinatorics}

Consider such a selection and write $c_i=q^{\alpha_i}$. Note that 
$\omega_1(c_1)=\dots=\omega_n(c_n)=\omega_0(q^0)=1$.
The above statement 
is verified by the juxtaposition of the following two lemmas.

\begin{lemma}
Let $\alpha_0=0$. If $F(c_0,c_1,\dots,c_n)\ne 0$, 
then $\alpha_i=\Delta_{i-1}$ for every $1\le i\le n$.
\label{first}
\end{lemma}

\begin{lemma}
If $\alpha_0\ne 0$, then
$({\partial^{m_0}F}/{\partial x_0^{m_0}})(c_0,\ldots,c_n)=0$.
\label{second}
\end{lemma}

One key to each is the following consequence of Claim \ref{trivi}.

\begin{lemma}
Suppose that $({\partial^{m_0}F}/{\partial x_0^{m_0}})(c_0,\ldots,c_n)\ne 0$.
Then for every $1\le t\le n-1$ there is at most one index $1\le i\le n$
such that $\alpha_i\in I_t$. 
\label{I_t}
\end{lemma}

\begin{proof}
Assume that, on the contrary, there is a pair $1\le i\ne j\le n$ such that
$\alpha_i,\alpha_j\in I_t$. Let $\alpha_j\ge \alpha_i$, then it follows from 
Claim \ref{trivi} that $\alpha_j-\alpha_i\ge k$.
The length of $I_t$ is $\gamma_t\in \{k,k+1\}$. Thus, it must be 
$\gamma_t=k+1$, $\alpha_i=\De_t-k$ and $\alpha_j=\De_t$. Consequently,
$t> n_0$, $i<j$ and $i\le n_0$. 
Therefore
$\De_t-\gamma_{\min\{t,i\}}+1=\De_t-k+1$ and $\alpha_i\not\in A_i$, a
contradiction.
\end{proof}

\bigskip
\noindent{\em Proof of Lemma \ref{first}.}
For every $1\le i\le n$ we have $\alpha_i\ge \alpha_0$, 
therefore $\alpha_i\ge \beta_{0i}=b$ by Claim \ref{trivi}.
Moreover, $k>0$ implies that $\alpha_1,\dots,\alpha_n$ are all distinct,
thus it follows from Lemma \ref{I_t} that each of the intervals $I_0,
I_1,\dots,I_{n-1}$ contains precisely one of them. Let $\pi\in\Symm_n$ denote
the unique permutation for which $\alpha_{\pi(1)}<\dots<\alpha_{\pi(n)}$,
then $\alpha_{\pi(i)}\in I_{i-1}$.
By Claim \ref{trivi} we have 
\[
\alpha_{\pi(i+1)}\ge \alpha_{\pi(i)}+\beta_{\pi(i),\pi(i+1)}+\chi(\pi(i)>\pi(i+1)).
\]
Consequently,
\[
\alpha_{\pi(n_0+1)}\ge b+kn_0+
\sum_{i=1}^{n_0}\chi(\pi(i)>\pi(i+1))\ge \Delta_{n_0}+
\sum_{i=1}^{n_0}\chi(\pi(i)>\pi(i+1)).
\]
Since $\alpha_{\pi(n_0+1)}\le \Delta_{n_0}$, it follows that $\alpha_{\pi(1)}=b$,
$\pi(1)<\dots<\pi(n_0+1)$, and $\beta_{\pi(i),\pi(i+1)}=k$ for $1\le i\le n_0$.
This in turn implies that $\pi(n_0)\le n_0$, thus $\pi(i)=i$ and 
$\alpha_i=\De_{i-1}$ for $1\le i\le n_0$.

Now for $n_0<i<n$ we have $\pi(i),\pi(i+1)>n_0$ and thus
$\beta_{\pi(i),\pi(i+1)}=k+1$. Restricting $\pi$ to the set $[n_0+1,n]$ and
starting with $\alpha_{\pi(n_0+1)}= \Delta_{n_0}$, a similar argument 
completes the proof.
\qed

\bigskip
\noindent{\em Proof of Lemma \ref{second}.}
Assume that, contrary to the statement, 
$({\partial^{m_0}F}/{\partial x_0^{m_0}})(c_0,\ldots,c_n)\ne 0$.
Write $S_t=\left[ \De_t-b+1, \De_t-b+ \beta_{t+1, 0}  \right] $. 
Since $\alpha_0\ne 0$ and $m_0<\omega_0(c_0)$, 
there is an index $0\le u\le n-\omega_0(c_0)$ 
such that $\alpha_0\in S_u\cap S_{u+1}\cap \dots \cap S_{u+m_0}$. That is,
\[
\De_{u+m_0}-b+1\le \alpha_0\le \De_u-b+ \beta_{u+1, 0}.
\]
Accordingly, if $\alpha_j$ lies in the interval
\[ 
T_{uj}=[\De_u-b+ \beta_{u+1, 0}-\beta_{j0},\De_{u+m_0}] 
\]
for some $1\le j\le n$, then
$\alpha_0-\beta_{j0}\le \alpha_j\le \alpha_0+\beta_{0j}-1$
and there is a term of the form $x_j-q^tx_0$ or $x_0-q^tx_j$ in $F$
which attains 0 when evaluated at the point $(c_0,\bc)$.
There cannot be more than $m_0$ such terms. It is implied by Lemma \ref{I_t}
that at most $n-1-u-m_0$ of the distinct numbers $\alpha_1,\dots,\alpha_n$
can lie in the interval $[\De_{u+m_0}+1,\De_{n-1}]$.

It follows that at least $u+1$ of the numbers $\alpha_j$ satisfy
$\alpha_j\le \De_u-b+ \beta_{u+1, 0}-\beta_{j0}-1$. This is clearly impossible
if $u+1\le n-m$, for then $\De_u-b+ \beta_{u+1, 0}-\beta_{j0}-1\le uk-1$ in
view of $n-m\le n_0$, and on the other hand the difference between any
two such $\alpha_j$ is at least $k$ in view of Claim \ref{trivi}. 
Thus, $u\ge n-m$ and $\beta_{u+1, 0}=a+1$. Consider
\[
\alpha_{\nu(1)}<\dots<\alpha_{\nu(u+1)}\le \De_u-b+ \beta_{u+1, 0}-\beta_{\nu(u+1),0}-1
\le \De_u-b.
\]
If $u\le n_0$, then it must be $\alpha_{\nu(i)}=(i-1)k$ and 
$\nu(1)<\dots<\nu(u+1)$, but then $\nu(u+1)\ge u+1>n-m$, $\beta_{\nu(u+1),0}=a+1$,
implying $\alpha_{\nu(u+1)}\in T_{u,\nu(u+1)}$, which is absurd.
This means that $u\ge n_0+1$. 
It is easy to see that $\alpha_{\nu(i+1)}-\alpha_{\nu(i)}\ge \gamma_{\nu(i)}$
for $i\le u$, thus  
$\alpha_{\nu(u+1)}\ge \sum_{i=1}^u\gamma_{\nu(i)}\ge \Delta_u-b.$
Therefore $\sum_{i=1}^u\gamma_{\nu(i)}= \Delta_u-b$, which implies that
$\{\nu(1),\dots,\nu(u)\}\supseteq \{1,\dots,n_0\}$. Consequently,
$\nu(u+1)\ge n_0+1>n-m$, which leads to a contradiction as before.
\qed

\subsection{The computation}

It only remains to evaluate
\begin{equation}
\frac{F_q(q^0,q^{\De_0},\dots,q^{\De_{n-1}};\bB)}{\psi_0\psi_1\dots\psi_n},
\label{fraction}
\end{equation}
where
\[
\psi_j={\prod}_{\alpha\in A_j\setminus\{\De_{j-1}\}}(q^{\De_{j-1}}-q^\alpha)
\]
for $j=1,\dots n$, and with the shorthand notation $\De_u^v=
\gamma_u+\dots+\gamma_v=\De_v-\De_{u-1}$,
\begin{equation}
\psi_0=\prod_{t=0}^{n-1}\prod_{\alpha=\De_1^t+1}^{\De_1^t+\beta_{t+1,0}}(1-q^\alpha)=
\prod_{j=1}^n \left[ \De_1^{j-1}+1,\De_1^{j-1}+\beta_{j0} \right]_q.
\label{denum_0}
\end{equation}
From now on, $[u,v]_q:=(1-q^u)\dots(1-q^v)=(q)_v/(q)_{u-1}$, with
$[u,u]_q$ abbreviated as $[u]_q$. Both the numerator and the denumerator
in (\ref{fraction})
is the product of factors in the form $\pm q^u(1-q^v)$ with some non-negative
integers $u,v$. More precisely, collecting factors of a similar nature
together we find that the numerator is the product of the factors
\begin{equation}
(-1)^{\gamma_0}\times 
q^{1+\dots+(\gamma_0-1)}\times 
\left[ \De_1^{j-1}+1,\De_0^{j-1}+\beta_{j0} \right]_q
\quad \mathrm{for} \quad 1\le j\le n,
\label{num_0j}
\end{equation}
\begin{equation}
(-1)^{\gamma_i}\times 
q^{\De_{i-1}+\dots+(\De_{i-1}+\gamma_i-1)}\times 
\left[ \De_i^{j-1}-\gamma_i+1,\De_i^{j-1} \right]_q
\quad \mathrm{for} \quad 1\le i< j\le n,
\label{num_ij(1)}
\end{equation}
and 
\begin{equation}
q^{\gamma_i\De_{i-1}}\times
\left[ \De_i^{j-1}+1,\De_i^{j-1}+\gamma_i \right]_q
\quad \mathrm{for} \quad 1\le i< j\le n.
\label{num_ij(2)}
\end{equation}
In the denominator, besides (\ref{denum_0}) we have the factors
\begin{equation}
(-1)\times 
\left[ \De_{j-1} \right]_q\times
\psi_{j_<} \times \psi_{j_=} \times \psi_{j_>}
\quad \mathrm{for} \quad 1\le j\le n,
\label{denum_j}
\end{equation}
where
\begin{equation}
\psi_{j_<}= \prod_{t=0}^{j-2}  (-1)^{\gamma_t}\times 
q^{(\De_{t}-\gamma_t+1)+\dots+\De_{t}}\times 
\left[ \De_{t+1}^{j-1},\De_{t+1}^{j-1}+\gamma_t-1 \right]_q,
\label{denum_j<}
\end{equation}
\begin{equation}
\psi_{j_=}=   (-1)^{\gamma_{j-1}-1}\times 
q^{(\De_{j-1}-\gamma_{j-1}+1)+\dots+(\De_{j-1}-1)}\times 
\left[ 1,\gamma_{j-1}-1 \right]_q,
\label{denum_j=}
\end{equation}
and
\begin{equation}
\psi_{j_>}= \prod_{t=j}^{n-1}  
q^{\gamma_j\De_{j-1}}\times
\left[ \De_j^{t}-\gamma_j+1,\De_j^{t} \right]_q.
\label{denum_j>}
\end{equation}

Now the powers of $-1$ and $q$ cancel out due to the simple identity
\[
n\gamma_0 + \sum_{1\le i<j\le n}\gamma_i =
n + \sum_{0\le t<j-1\le n-1}\gamma_t + \sum_{1\le j\le n}(\gamma_{j-1}-1)
\]
and the somewhat more subtle
\begin{multline*}
n \binom{\gamma_0}{2} +
\sum_{1\le i<j\le n} \left( 2\gamma_i\De_{i-1} + \binom{\gamma_i}{2} \right) \\=
\sum_{0\le t<j-1\le n-1} \left( \gamma_t\De_{t} - \binom{\gamma_t}{2} \right) +
\sum_{j=1}^n \left( (\gamma_{j-1}-1)\De_{j-1} - \binom{\gamma_{j-1}-1}{2} \right) +
\sum_{0\le j-1<t\le n-1} \gamma_j\De_{j-1}.
\end{multline*}

It remains to deal with the factors of the form $[u,v]_q$. Those from
(\ref{num_ij(1)}) and (\ref{denum_j>}) cancel out. Those from 
(\ref{num_0j}) and (\ref{denum_0}) yield
\begin{equation}
\prod_{j=1}^n \frac{(q)_{\De_{0}^{j-1}+\beta_{j0}}}{(q)_{\De_{1}^{j-1}+\beta_{j0}}} =
\prod_{j=0}^{n-1}
\frac{(q)_{a+b+kj+\chi(j>n_0)(j-n_0)+\chi(j\ge n-m)}}
{(q)_{a+kj+\chi(j>n_0)(j-n_0)+\chi(j\ge n-m)}}.
\label{eleje}
\end{equation}
As for the rest, the contribution 
from (\ref{num_ij(2)}) and (\ref{denum_j<}) with the substitution $t+1=i$ gives
\begin{align}
\prod_{1\le i<j\le n} \frac{\left[ \De_i^{j-1}+1,\De_i^{j-1}+\gamma_i \right]_q}
{\left[ \De_{i}^{j-1},\De_{i}^{j-1}+\gamma_{i-1}-1 \right]_q}&=
\prod_{1\le i<j\le n} \frac{\left[ \De_i^{j-1},\De_i^{j-1}+\gamma_i \right]_q
\cdot \left[ \De_{i}^{j-1}+\gamma_{i-1} \right]_q}
{\left[ \De_{i}^{j-1} ,\De_{i}^{j-1}+\gamma_{i-1} \right]_q
\cdot \left[ \De_{i}^{j-1} \right]_q} \notag \\
&= \prod_{j=2}^n \frac{\left[ \De_1^{j-1},\De_1^{j-1}+\gamma_1 \right]_q}
{\left[ \De_{1}^{j-1} ,\De_{1}^{j-1}+\gamma_{0} \right]_q} \times \Psi \times
\prod_{1\le i<j\le n} \frac{\left[ \De_{i-1}^{j-1} \right]_q}
{\left[ \De_{i}^{j-1} \right]_q} \notag \\
&=\prod_{j=2}^n \frac{(q)_{\De_{1}^{j-1}+\gamma_1}}{(q)_{\De_{1}^{j-1}+\gamma_0}}
\times \Psi \times
\prod_{j=2}^n \frac{\left[ \De_{j-1} \right]_q}{1-q^{\gamma_{j-1}}} \label{utolso}
\end{align}
in the first place, where the factor
\[
\Psi=\prod_{j=n_0+2}^n \left[ \De_{n_0+1}^{j-1}+\gamma_{n_0+1} \right]_q =
\prod_{j=2}^{n-n_0} (1-q^{(k+1)j})
\]
only occurs when $n_0>0$. Combining (\ref{utolso}) with the contribution
of the factors $\left[ \De_{j-1} \right]_q=1-q^{\De_{j-1}}$ from (\ref{denum_j})
and the factors $\left[ 1,\gamma_{j-1}-1 \right]_q=(q)_{\gamma_{j-1}-1}$
from (\ref{denum_j=}), shifting indices we obtain
\[
\prod_{j=1}^{n-1} \frac{(q)_{\De_{1}^{j}+\gamma_1}}{(q)_{\De_{1}^{j}+\gamma_0}} \times
\prod_{j=0}^{n-1} \frac{1}{(q)_{\gamma_j}} \times
\left( \prod_{j=2}^{n-n_0} (1-q^{(k+1)j}) \right)^{\chi(n_0>0)},
\]
in agreement with
\begin{equation}
\frac{(q)_{kj+\chi(j>n_0)(j-n_0)+k}}
{(q)_{b+kj+\chi(j>n_0)(j-n_0)}(q)_k}\times
\prod_{j=1}^{n-n_0} \frac{1-q^{(k+1)j}}{1-q^{k+1}}.
\label{vege}
\end{equation}
Putting together (\ref{eleje}) and (\ref{vege}) completes the proof 
of Theorem \ref{a-f}.

\bigskip
\noindent
{\em Remark.} For all the identities considered in this paper, the formulas
exhibit, apart from some minor deviations, quite a similar pattern,
and it is more or less clear from the above argument, why it is so. 
We do not elaborate on this here, but the motivated reader may come up
with other families of matrices $\bB$ for which a similar proof strategy
might work. We believe that the details given above can be useful in such
a quest.

\subsection{A rationality result}

It is possible to prove Theorem \ref{a-f} based solely on Lemma \ref{interpol};
in fact this is how our result was originally obtained. It involves the
same combinatorics applied when $k\ge a+1$, in which case $A_0$ is an
ordinary set. The extension of the result that includes all non-negative 
integers $k$ depends on the following rationality lemma, inspired by
\cite[Proposition 2.4]{GLXZ}. 

\begin{lemma}
Fix nonnegative integers $r_i,s_i$ for $1\le i\le n$, satisfying
$\sum r_i=\sum s_i$.
There is a rational function $Q=Q(z)\in \Q(q)(z)$ that depends only on
$n$ and the numbers $r_i,s_i$, such that
\[
\CT\Bigg[ \frac{x_1^{r_1}\dots x_n^{r_n}}{x_1^{s_1}\dots x_n^{s_n}}\D_q(\bx;k) \Bigg]=
Q(q^k) \frac{(q)_{nk}}{(q)_k^n}.
\] 
\label{rational}
\end{lemma}

\noindent
Expanding the degree zero part of  
\[
\prod_{j=1}^n(qx_j)_{a+\chi(j\le m)}(1/x_j)_b
\prod_{n_0<i<j\le n} (1-q^kx_i/x_j)(1-q^{k+1}x_j/x_i)
\] 
into a sum of monomial terms and applying the above lemma to each such term
individually, we find that there is a rational function $R\in \Q(q)(z)$
depending only on the parameters $n,m,n_0,a,b$ such that
\[
\CT[\L_q(x_0,\bx;\bB_{\A\Fo})]=R(q^k) \frac{(q)_{nk}}{(q)_k^n}.
\]
Reorganizing the formula in Theorem \ref{a-f} in the form
\[
P(q^k) \frac{(q)_{nk}}{(q)_k^n}
\]
with a function $P\in \Q(q)(z)$ which also depends only on $n,m,n_0,a,b$,
the theorem established for $k\ge a+1$ yields $P=R$, which in turn implies
the full content of the result. 
\medskip

It only remains to prove Lemma \ref{rational}, and this is executed 
with yet another application of Lemma \ref{interpol}. Since a similar ---
in fact more general --- result was found recently by Doron Zeilberger and
his able computer \cite{EZ}, we only give a brief account.
As the $k=0$ case is trivial, we will assume $k>0$.

\bigskip
\noindent{\em Proof of Lemma \ref{rational}.}
The constant term of $\D_q(\bx;k)$ equals the coefficient of
$\prod x_i^{(n-1)k}$ in the polynomial
\[
F(\bx)=\prod_{1\le i<j\le n} \Bigg(\prod_{t=0}^{k-1}(x_j-q^tx_i)
\times \prod_{t=1}^{k}(x_i-q^tx_j) \Bigg).
\]
Set $C_i=\{q^{\alpha_i}\mid \alpha_i\in [0,(n-1)k]\}$. Then $F(\bc)=0$
for every $\bc\in C_1\times\dots\times C_n$ except when
$c_i=q^{(i-1)k}$ for every $i$. According to Lemma \ref{interpol},
\[
\CT[\D_q(\bx;k)]=\frac{F(q^0,q^k,\dots,q^{(n-1)k})}{\psi_1\psi_2\dots\psi_n}=
\frac{(q)_{nk}}{(q)_k^n}\quad \mathrm{where}\quad
\psi_i={\prod}_{0\le j\le (n-1)k, j\ne (i-1)k}(q^{(i-1)k}-q^j).
\]
We compare this product to the constant term in the lemma, which equals the
coefficient of $\prod x_i^{(n-1)k+s_i}$ in the polynomial
$F^*(\bx)=x_1^{r_1}\dots x_n^{r_n}F(\bx)$. Accordingly we set
$C_i^*=\{q^{\alpha_i}\mid \alpha_i\in [0,(n-1)k+s_i]\}$ and note that
for $\bc\in C_1^*\times\dots\times C_n^*$ we have $F^*(\bc)\ne 0$ if and only if
the exponents $\alpha_i$ are all distinct and 
\[
\alpha_{\pi(i+1)}\ge \alpha_{\pi(i)}+k+\chi(\pi(i)>\pi(i+1))
\]
holds for $1\le i\le n-1$ with the unique permutation $\pi=\pi_{\bc}\in \Symm_n$
satisfying $\alpha_{\pi(1)}<\dots<\alpha_{\pi(n)}$. Consequently,
$\alpha_i=(\pi^{-1}(i)-1)k+\epsilon_i$ for some 
$\epsilon_i=\epsilon_i(\bc)\in [0,s_{\pi(n)}]$.

Set $\Ce=\{\bc\in C_1^*\times\dots\times C_n^*\mid F^*(\bc)\ne 0\}$, and write
$s=\max s_i$. It follows that 
\[
|\Ce|\le n!\binom{s+n}{n}.
\]
Moreover, the set 
$\S=\{(\pi_{\bc},\epsilon_1(\bc),\dots\epsilon_n(\bc))\mid \bc\in \Ce\}$
is independent of $k$; it depends only on $n$ and the numbers $s_i$.
It follows from Lemma \ref{interpol} that, using the notation $\tau=\pi^{-1}$,
\[
\CT\Bigg[ \frac{x_1^{r_1}\dots x_n^{r_n}}{x_1^{s_1}\dots x_n^{s_n}}\D_q(\bx;k) \Bigg]=
\sum_{\bc\in \Ce} \prod_{i=1}^n q^{((\tau(i)-1)k+\epsilon_i)r_i}
\frac{F(\dots,q^{(\tau(i)-1)k+\epsilon_i},\ldots)}{\psi_1^*\psi_2^*\dots\psi_n^*}
\]
where
\[
\psi_{\pi(i)}^*={\prod}_{0\le j\le (n-1)k+s_{\pi(i)}, j\ne (i-1)k+\epsilon_{\pi(i)}}
(q^{(i-1)k+\epsilon_{\pi(i)}}-q^j).
\]
One readily checks that for each $\Sigma=(\pi,\epsilon_1,\dots,\epsilon_n)
\in \S$ there exist rational functions $Q_i\in \Q(q)(z)$ that depend only on
$n$, the numbers $r_j,s_j$ and the sequence $\Sigma$, such that
\[
\prod_{i=1}^n q^{((\tau(i)-1)k+\epsilon_i)r_i}=Q_0(q^k),\quad
\frac{\psi_i}{\psi_{\pi(i)}^*}=Q_i(q^k)\quad \mathrm{and}\quad
\frac{F(\dots,q^{(\tau(i)-1)k+\epsilon_i},\ldots)}{F(q^0,q^k,\dots,q^{(n-1)k})}=
Q_{n+1}(q^k).
\]
The result follows.
\qed

\bigskip


\bigskip
\noindent
{\bf Acknowledgments}
\bigskip


We gratefully acknowledge the support of the Australian Research
Council, ERC Advanced Research Grant No. 267165, 
OTKA Grants 81310 and 100291,
RFBR Grants 11-01-00677-a, 13-01-00935-a and
13-01-12422-ofi-m, President of Russia Grant MK-6133.2013.1,
and Russian Government Mega-Grant Project 11.G34.31.0053.






\end{document}